\documentclass[12pt]{amsart}
\usepackage{amsmath, amsthm}
\usepackage{amssymb}
\usepackage{latexsym}
\usepackage{amsfonts}
\usepackage[dvips]{graphicx}
\usepackage{color}
\usepackage{graphicx,epstopdf}
\usepackage{tikz-cd}

\renewcommand{\a}{\alpha}

\usepackage{float}
\usepackage[section]{placeins}

\newtheorem{theorem}{Theorem}[section]
\newtheorem{theorem*}[theorem]{Theorem}
\newtheorem{lemma}[theorem]{Lemma}

\newtheorem{proposition}[theorem]{Proposition}
\newtheorem{corollary}[theorem]{Corollary}
\newtheorem{example}{Example}
\theoremstyle{remark}
\newtheorem{remark}[theorem]{Remark}

\numberwithin{equation}{section}

\newcommand{\p}{\partial}

\newcommand{\Der}{\text{Der}}

\renewcommand{\a}{\alpha}

\newcommand{\R}{\mathbb R}

\newcommand{\C}{\mathbb C}

\newcommand{\A}{\mathcal A}

\newcommand{\N}{\mathbb N}

\def\beq#1#2\eeq{%
        \begin{equation}%
        \label{#1}%
            #2%
        \end{equation}%
    }

\usepackage[hmargin={25mm,25mm}]{geometry}

\usepackage{color}
\usepackage{graphicx}
\usepackage{epstopdf}

\begin{document}

\title{Free reflection multiarrangements and quasi-invariants}

\author{T. Abe, N. Enomoto, M. Feigin, M. Yoshinaga}

\email{abe.takuro.831@m.kyushu-u.ac.jp,
enomoto-naoya@uec.ac.jp, misha.feigin@glasgow.ac.uk,
yoshinaga@math.sci.osaka-u.ac.jp}

\begin{abstract}
To a complex reflection arrangement with an invariant multiplicity function one can relate the space of logarithmic vector fields and the space of quasi-invariants, which are both modules over invariant polynomials. 
We establish a close relation between these modules. Berest--Chalykh freeness results for the module of quasi-invariants lead to new free complex reflection multiarrangements. K. Saito's primitive derivative gives a linear map between certain spaces of quasi-invariants.

We also establish a close relation between non-homogeneous quasi-invariants for root systems  and logarithmic vector fields for the extended Catalan arrangements. As an application, we prove the freeness of Catalan arrangements corresponding to the non-reduced root system $BC_N$.  

\end{abstract}

\maketitle

\section{Introduction}

Let us consider firstly a finite real reflection group $W$, and let $\mathcal A$ be the corresponding Coxeter arrangement of hyperplanes in a vector space $V$, and $m\colon {\mathcal A} \to \N$ a $W$-invariant multiplicity function. To this data one can associate a ring of $m$-{\it quasi-invariant} polynomials $Q_m = Q_{m, {\mathcal A}} \subseteq SV^*$. 
Such a ring appeared first in the work of Chalykh and Veselov on quantum Calogero--Moser systems \cite{CV}. In that context the ring $Q_m$ was isomorphic to the commutative ring of differential operators containing the second order Hamiltonian of Calogero--Moser type at integer coupling parameters. This ring contains the ring of $W$-invariant polynomials $(SV^*)^W$, and it is also a $W$-module. 

Algebraic properties of rings of quasi-invariants were investigated by Feigin and Veselov in \cite{FeiV} where the term was introduced. It was shown that in the  two-dimensional case with constant multiplicity function the rings were graded Gorenstein and it was conjectured that the same holds in general. In particular, the rings were expected to be free modules over invariants $(SV^*)^W$. This was proven by Etingof and Ginzburg in \cite{EG}. Hilbert series for related $m$-harmonic polynomials were found by Felder and Veselov in \cite{FV}, which led to the results on Hilbert series of quasi-invariants and the statement that quasi-invariants are graded Gorenstein \cite{EG}.  Another proof of Cohen--Macaulay property on free generation was found by Berest, Etingof and Ginzburg in \cite{BEG}. It is based on the observation that quasi-invariants form a module over the spherical subalgebra of the rational Cherednik algebra $e{\mathcal H}_c(W)e$ at integer value of the parameter $c=m$. The module of quasi-invariants $Q_m$ can be decomposed as a direct sum of submodules corresponding to different irreducible representations of $W$, and any representation of the algebra $e{\mathcal H}_m(W)e$ from category $\mathcal O$ can be obtained as a direct sum of such subrepresentations of $Q_m$.

Quasi-invariants $Q_m$ for an arbitrary finite complex reflection group $W$ acting by compositions of reflections in a complex vector space $V$ were introduced by Berest and Chalykh in \cite{BC}. In this case $Q_m$ is also a free module over the subspace of $W$-invariants $(SV^*)^W$ but in general quasi-invariants do not form a ring. It was shown in \cite{BC} that quasi-invariants form a representation of the corresponding spherical rational Cherednik algebra $e{\mathcal H}_m(W)e$, which implied the freeness statement,  and the Hilbert series of $Q_m$ were found. 
Quasi-invariants are defined in \cite{BC} for more general vector-valued $W$-invariant multiplicity functions so that up to $n_H-1$ values are associated with $H\in \A$, where $n_H$ is the size of cyclic subgroup of $W$ which fixes $H$ pointwise. In this paper we will deal with a special case where we assume that all these values are the same.

Furthermore, vector-valued quasi-invariants ${\bf Q}_m(\tau)$ were introduced and studied in \cite{BC} for any $W$-module $\tau$. These quasi-invariants are modules over full rational Cherednik algebra ${\mathcal H}_m(W)$.

On the other hand there is a theory of free hyperplane arrangements. The notion goes back to the work \cite{Sa1} of K. Saito. 
The notion of logarithmic vector fields for arrangements with integral multiplicity function was introduced by  Ziegler \cite{Z}. 
The freeness of any Coxeter multiarrangement with a constant multiplicity function $m$ is proved in  \cite{T}, and that with a general invariant multiplicity functions is proved in \cite{ATWaka}. 

In the case of complex reflection multiarrangements freeness was studied by Hoge, Mano,  R\"ohrle and Stump in \cite{HMRS}. The authors considered two types of multiplicity functions: $r_1 = m n +1$ and  $r_2 = m n$, where
$n(H) = n_H$ is the order of stabilizer of the hyperplane $H \in {\mathcal A}$ and $m \in \N$ is constant. The corresponding multiarrangements were shown to be free for well-generated complex reflection groups $W$ and for the infinite series $W=G(r, p, N)$, and degrees of free basis elements were determined. 
Since all multiarrangements on the plane are free the only case with these multiplicities when freeness remained open was the exceptional group $W=G_{31}$, which is not a well-generated group of rank 4. 

In this paper we relate theory of quasi-invariants with theory of logarithmic vector fields for reflection arrangements.
We consider the logarithmic vector fields $D_m$ and $\widetilde D_m$ for multiplicity functions $r_1$ and $r_2$ as above, respectively, where $m\colon {\mathcal A}\to \N$ is an arbitrary $W$-invariant function.
 Let $Q_m^{\tau}$ be the isotypic component in $Q_m$ of an irreducible $W$-module $\tau$. Then a relation with the logarithmic vector fields is given by the following isomorphism:
\begin{equation}
\label{isom1}
Q_m^{V^*}\cong D_m^W\otimes V^*,
\end{equation}
which is an isomorphism of $(SV^*)^W\otimes \C W$-modules. Also for quasi-invariants 
${\bf Q}_m(V)$ 
with values in reflection representation $V$ we have an isomorphism
\begin{equation}
\label{isom2}
{\bf Q}_m(V)\cong \widetilde D_m.
\end{equation}
This is an isomorphism of both $S V^*$ and $\C W$-modules but the actions of these algebras do not commute. 
As a consequence of isomorphisms \eqref{isom1}, \eqref{isom2} we derive freeness of complex reflection arrangements with multiplicities $r_1, r_2$ for any complex reflection group $W$ and $W$-invariant function $m$ (Section \ref{IsomFree}). The Hilbert series of quasi-invariants were determined in \cite{BC} in terms of KZ twists which are special permutations of irreducible representations of $W$ defined via monodromy of Knizhnik--Zamolodchikov connections \cite{Op}.
This leads to the corresponding expressions of degrees of free generators of logarithmic vector fields (cf. \cite{HMRS}).

Another consequence of isomorphisms \eqref{isom1}, \eqref{isom2} is the action of the primitive vector field on the corresponding quasi-invariants.
We consider this action in Section \ref{quasi-prim}.

In Section \ref{affine} we explain relation between non-homogeneous (trigonometric) version of quasi-invariants and logarithmic vector fields for the extended Catalan arrangements associated with Weyl groups. In the case of the non-reduced root system $BC_N$ we get new free arrangements.

\section{Preliminaries on quasi-invariants and reflection arrangements}

Let $W$ be an irreducible finite complex reflection group and $V\cong \C^N$, $N\in \N$ its reflection representation. Let $e_1, \ldots, e_N$ be the standard basis in $V$.  Let $\A\subseteq V$ be the corresponding reflection arrangement of hyperplanes. For each $H \in \A$ let $n_H\in \N$ be the size of the stabiliser $W_H$ of $H$ in $W$.   Let $\alpha_H \in V^*$ be such that the hyperplane $H$ is given as $H =\{x\in V\colon \a_H(x)=0\}$.

Let $m\colon \A \to \N$ be a $W$--invariant (multiplicity) function. We denote $m(H)$ as $m_H$, $H\in \A$.
Likewise we consider the order of the stabilizer $W_H$ as the multiplicity function $n\colon {\mathcal A}\to \N$,   $n_H = n(H)$.




Let $\Der$ be the set of $\C$-linear maps $L\colon SV^* \to SV^*$ that satisfy the Leibniz rule: $L(\varphi \psi)=L(\varphi)\psi+\varphi L(\psi)$
 for all $\varphi,\psi \in SV^*$. Pick the standard basis $x_1, \ldots ,x_N$ in $V^*$ such that $x_i(e_j)=\delta_{ij}$. 
 We can consider a derivation $L \in \Der$ as a vector field on $V$ with polynomial coefficients, 
 that is we can write $L=\sum_{i=1}^{N}f_i\partial_i$,  where $f_i \in SV^* \cong \C[x_1, \ldots ,x_N]$ and $\partial_i=\partial_{x_i}$.  
Since the group $W$ acts on $SV^*$, we say that $L \in \Der$ is $W$-invariant if and only if $L(w\varphi)=w(L\varphi)$ for all $w \in W$ and $\varphi \in SV^*$.
In other words, $W$ acts on $\Der$ by $(wL)(\varphi)=w(L(w^{-1}\varphi))$ for $\varphi \in SV^*$ and $w \in W$. 
We denote by $\Der^W$ the set of $W$-invariant derivations in $\Der$.

Before going further, define the logarithmic vector fields of an (affine) arrangement
$\mathcal{B}$ with the multiplicity function $r\colon \mathcal{B} \to \N$ as follows:
$$
D(\mathcal{B}, r) = \{L \in \Der\colon L(\alpha_H) \in \alpha_H^{r(H)} SV^* \,  \text { for any } \,  H \in \mathcal{B}\}. 
$$
A pair $(\mathcal{B},r)$ is called a multiarrangement, and its logarithmic vector field 
$D(\mathcal{B},r)$ is an $SV^*$-graded module, which is  not free in general. We say that the multiarrangement $(\mathcal{B},r)$ is free if $D(\mathcal{B},r)$ is free as an $SV^*$-graded module. We say that a derivation $0 
\neq \theta \in \Der$ is 
homogeneous of degree $\deg\theta=d$ if $\theta(\alpha)$ is either zero or a homogeneous polynomial of degree $d$ for all $\alpha \in V^*$. When $(\mathcal{B},r)$ is central and free, we can choose a homogeneous basis $\theta_1,\ldots,\theta_N$ for $D(\mathcal{B},r)$. Then the exponents 
of a free multiarrangement $(\mathcal{B},r)$
are defined as a multiset $(\deg \theta_1,\ldots,\deg\theta_N)$. 

Now let us come back to a reflection multiarrangement $(\A,r)$. We will be interested in multiplicity functions $r$ of the form $r = m n + 1$ and $r = m n$. We denote
$$
D_m = D({\mathcal A}, mn +1), \quad {\widetilde D}_m = D({\mathcal A}, m n).
$$
Note that $(wL)(\alpha_H)=w(L(w^{-1}\alpha_H))=w(\alpha_{w^{-1}H}^{r(w^{-1}H)}p)=\alpha_H^{r(w^{-1}H)}w(p)$ for some $p \in SV^*$. Since $mn+1$ and $mn$ are  $W$-invariant multiplicity functions, 
$W$ acts on $D_m$ and $\widetilde{D}_m$.
Let $D_m^W$ be the corresponding space of $W$-invariant elements from $D_m$.

Consider elements
$$
e_{H, i} =  \sum_{w \in W_H} (\det w)^i w \in \C W_H, \quad i=1, \ldots, n_H-1,
$$
where $\det w$ is the determinant of the corresponding map $w\colon V \to V$.

Following \cite{BC} we have the following definition. A polynomial $p\in SV^*$ is called  {\it $m$-quasi-invariant} if the polynomial $
e_{H, i} (p) 
$
is divisible by $\alpha_H^{n_H m_H}$ for all $H\in \A$, $1\le i \le n_H-1$. 
This condition can also be reformulated in the following way.
\begin{proposition}\label{propidemp}
Let $s_H$ be a generator of the cyclic group $W_H$. A polynomial $p$ is m-quasi-invariant if and only if 
\begin{equation}
\label{quasicond}
\a_H^{m_H n_H} | (1- s_H) p
\end{equation}
for all $H\in \A$. 
\end{proposition}
\begin{proof}
It follows by applying $s_H$ that $\a_H^{n_H m_H}| (1-s_H)p$ if and only if $\a_H ^{n_H m_H}| (1-s_H^k)p$ for all $k \in \N$.
Consider $(n_H-1)\times (n_H-1)$ matrix $A=(a_{ij})_{i,j=1}^{n_H-1}$, where $a_{ij} = \xi^{ij}$, and $\xi = \det s_H$ is a primitive root of order $n_H$. Let 
$$
e=(e_{H,1}, \ldots, e_{H, n_H-1}), \quad s = (1-s_H, \ldots, 1-s_H^{n_H-1}).
$$
It is easy to see that $A s^T =  - e^T$. The statement follows since $\det A \ne 0$. 
\end{proof}
Note that condition \eqref{quasicond} is equivalent to $\a_H^{m_H n_H+1} | (1- s_H) p$. We denote the space of all $m$-quasi-invariants as $Q_m$. 

\begin{proposition}\label{qmring}
The space $Q_m$ is a ring.
\end{proposition}
\begin{proof}
Let $p, q \in Q_m$. Then 
$$
(1- s_H)(p q) = q (1 - s_H)(p) +s_H(p)(1-s_H)(q),
$$
which is divisible by $\a_H^{n_H m_H}$.
\end{proof}

We note that the space $Q_m$ is a special case of quasi-invariants from \cite{BC}. More general spaces of quasi-invariants were considered in \cite{BC} which do not necessarily form rings. 

The space of quasi-invariants $Q_m$ is graded by degree. The same applies to the isotypic component  $Q_m^{\tau}$ of any irreducible $W$-module $\tau$.
 Hilbert series of the isotypic component $Q_m^{\tau}$ is  found in \cite{BC}.  It is defined via the value $c_{\tau^*}(m)$ of the central element 
 $$
 z =\sum_{H\in \A} m_H \sum_{i= 1}^{n_H-1} e_{H, i} \in \C W
 $$ 
 on the module $\tau^*$.  Let us state this precisely for the case $\tau = V^*$. 
Let the generator $s_H\in W_H$ have an  eigenvalue $\xi^{}\ne 1$ when acting on $V$. Then $\mbox{Tr } s_H|_{V} = \xi^{} + N -1$ and
$$
\mbox{Tr } z|_{V} =\sum_{H\in \A}  m_H \sum_{i=1}^{n_H-1} \sum_{k=1}^{n_H} \xi^{k i} (\xi^{k}+N-1) =\sum_{H\in \A}   m_H \sum_{i=1}^{n_H-1} \sum_{k=1}^{n_H}  \xi^{(i+1)k} = \sum_{H\in \A} m_H n_H.
$$
Therefore
\begin{equation}
\label{centrec}
c_V(m) = z|_{V}  = \frac{1}{N} \sum_{H\in \A} m_H n_H.
\end{equation}
Then the Hilbert series of the isotypic component $Q_m^{V^*}$ of representation $V^*$ in $Q_m$  has the form 
\begin{equation}
\label{hilb1}
P(t) = t^{c_V(m)} \chi_{\Phi(V)}(t),
\end{equation}
where $\chi_\tau(t)$ is the Hilbert series of the $W$-module $\tau^*$ in the polynomial ring  $SV^*$, $\chi_\tau (t) = P((SV^*\otimes \tau)^W, t)$, and $\Phi(V)=\mbox{kz}_{-m}(V)$ is the KZ-twist \cite{BC} (see also \cite{Op}, \cite{Ma}).

For further use, let us introduce other quasi-invariants. 
Berest and Chalykh also introduced quasi-invariants ${\bf Q}_m(\tau)$ with values in a $W$-module $\tau$ \cite{BC}. For the multiplicity function $m$ these are elements  $\varphi \in SV^* \otimes \tau$ such that
$$
(1\otimes e_{H, i})(\varphi) \equiv 0  \bmod  \langle \a_H\rangle^{m_H n_H} \otimes \tau,
$$
for all $H\in \A$  and $i=1, \ldots, n_H-1$. It is also shown in \cite{BC} that ${\bf Q}_m(\tau)\cong M_m(\Phi(\tau))$, where $M_m(\Phi(\tau))$ is the standard module over the corresponding rational Cherednik algebra, which is free over $SV^*$. In the case $\tau = V$ the space ${\bf Q}_m(\tau)$  as a module over $SV^*$  is generated by $N$ elements $\varphi^{(k)} = \sum_{i=1}^N f_i^{(k)} \otimes e_i$, $k=1, \ldots, N$, where $f_i^{(k)}$ are homogeneous polynomials of degree $c_V(m)$ given by formula \eqref{centrec} (see \cite{BC}).

\section{Isomorphisms and freeness results}
\label{IsomFree}

As in the previous section, let $W$ be an 
irreducible finite complex reflection group, $\A$ its reflection hyperplane arrangement, 
$D_m$ the logarithmic vector fields of $\A$, and $Q_m$ the quasi-invariant ring with $W$-invariant multiplicity $m$. 
We consider the derivation $L=\sum_{i=1}^N f_i \p_i \in \Der$. 
Then
$
w \sum_{i=1}^N f_i(x) \p_i = \sum_{i=1}^N f_i(w^{-1}(x)) w \p_i
$ 
for any $w\in W$. 

\begin{proposition}\label{prop1}
For $L=\sum_{i=1}^N f_i \p_i \in \Der^W$, we have $L \in D_m^W$ if and only if $f_i \in Q_{m}$ for all $1\le i \le N$.
\end{proposition}
\begin{proof}
Let $s_H$ be a generator of $W_H$ for $H \in \mathcal A$. We take $\alpha_H^\vee \in V$ such that 
$\mathrm{Im}(s_H-\mathrm{id})=\langle \alpha_H^\vee \rangle$ and 
$s_H(v) = v -  \alpha_H(v)\alpha_H^\vee$ for $v \in V$.
For any $\varphi \in V^*$ we have  $(s_H^{-1}\varphi)(v)=\varphi(s_Hv)=\varphi(v)-\alpha_H(v)\varphi(\alpha_H^\vee)$, hence 
$s_H^{-1}(\varphi)=\varphi-\varphi(\alpha_H^\vee)\alpha_H$. 
By the $W$-invariance of $L$, we obtain
\[
s_H^{-1}(L(\varphi))-L(\varphi)
=L(s_H^{-1}\varphi)-L(\varphi)=
-\varphi(\alpha_H^\vee)L(\alpha_H).
\]
Put $\alpha_H^\vee=\sum_{i=1}^{N}c_je_j$ and $\varphi=x_i$, then $L(\varphi)=f_i$ and $s_H^{-1}f_i-f_i=-c_iL(\alpha_H)$ for all $1 \leq i \leq N$.
We have $\alpha_H^{m_Hn_H} | L(\alpha_H)$ if and only if $\alpha_H^{m_Hn_H} | (1-s_H)f_i$ for all $1 \leq i \leq N$. The claim follows.
\end{proof}

Note that a $W$-invariant  vector field $L$ is determined by any of its components. More exactly we have the following statement.
\begin{lemma}\label{onecompall}
Let $L=\sum_{i=1}^N f_i \p_i \in \Der^W$. 
If there exists $0 \neq \beta \in V^*$ such that $L(\beta)=0$, then $f_i=0$ for all $1 \leq i \leq N$.
\end{lemma}
\begin{proof}
By the $W$-invariance of $L$ and the assumption $L(\beta)=0$, we have 
\[
0=s_H^{-1}(L(\beta))=L(s_H^{-1}\beta)=L(\beta-\beta(\alpha_H^\vee)\alpha_H)=
-\beta(\alpha_H^\vee)L(\alpha_H).
\]
Then $L(\alpha_H)=0$ if $\beta(\alpha_H^\vee) \neq 0$. 
Note that the $\C$-span $\langle \alpha_H^\vee \ | \ \beta(\alpha_H^\vee)\neq 0 \rangle$ coincides with $V$. Hence $L=0$ on $V^*$. 
\end{proof}

\begin{proposition}\label{prop2}
Let $0 \neq L=\sum_{i=1}^N f_i \p_i  \in D_m^W$. Consider the linear span 
$$
V_f=\langle f_i\colon  1\le i \le N \rangle \subseteq SV^*.
$$
 Then $V_f \cong V^*$ as $W$-modules. The isomorphism  $\psi\colon V^* \to V_f$ is given by $\psi (\xi) = L(\xi)$ for $\xi \in V^*$. 
 \end{proposition}
 \begin{proof}
 Note that 
 $\psi(\xi)=\sum_{i=1}^{N}f_i(\partial_i(\xi)) \in V_f$. 
Since $L(x_i)=f_i$, $\psi$ is surjective. By  Lemma 3.2, $\psi$ is injective. 
 By the $W$-invariance of $L$, 
 the morphism $\psi$ is $W$-invariant. 
 \end{proof}


Recall that $Q_m$ is a $W$-module as well as a free module over $(SV^*)^W$ \cite{BC}. Let $Q_m^{V^*}$ be the isotypic component of the $W$-module $V^*$ in $Q_m$. Then
$$
Q_m^{V^*} \cong \text{Hom}_W(V^*, Q_m)\otimes V^*
$$
as $(S V^*)^W\otimes \C W$-modules.
The following theorem follows from Propositions \ref{prop1}, \ref{prop2}.

\begin{theorem}\label{mth1}
There is an isomorphism 
$$
\Theta\colon \text{\rm Hom}_W(V^*, Q_m) \cong D_m^W
$$
of $W$-modules. Under this isomorphism
$$
\Theta(\varphi) = \sum_{i=1}^N \varphi(x_i) \partial_i, 
$$
where $\varphi \in \text{\rm Hom}_W(V^*, Q_m)$. For the inverse map we have $\Theta^{-1}(\sum_{i=1}^N f_i \partial_i)(x_k)=f_k$, $1\le k \le N$.
\end{theorem}

In the case of symmetric group the basis of $Q_{m}^{V^*}$ is found in \cite{BM} which generalises construction for the basis elements of the lowest degree from \cite{FV}. These results and Theorem \ref{mth1} have the following implication.
\begin{corollary}
Let $W=S_{N+1}$, which acts in its reflection representation in the space $V=\{(x_1, \ldots, x_{N+1}): \sum x_i =0$\}; $\A=\{V\cap \{x_i=x_j\}\colon 1\le i < j \le N+1\}$.  Let vector fields
$$
L^{(k)}=\sum_{i=1}^{N+1} f_i^ {(k)}\p_i 
$$
be given by
$$
f_i^{(k)} = \sum_{j=1}^{N+1} \int_{x_i}^{x_j} t^k \prod_{s=1}^{N+1} (t-x_s)^m dt,
$$
where $0\le k \le N-1$. Then vector fields $L^{(k)}$ form a basis of $D_m^W$. 
\end{corollary}


Theorem \ref{mth1} provides $N$ vector fields 
$$
L^{(k)} = \sum_{i=1}^N  q_i^{(k)} (x) \p_{x_i}
$$
that are $W$-invariant, and 
$q_i^{(k)} \in Q_m^{V^*}$ for all $i,k = 1,\ldots, N$, so that $L^{(k)}$ form a basis of $D_m^W$ over $(SV^*)^W$. Let us consider the $N\times N$ matrix $\mathcal Q$ whose $(ij)$-th entry is $q_i^{(j)}$.
Further, one can assume that polynomials $q_i^{(k)}$ are homogeneous and let $d_k = \deg q_i^{(k)}$. It follows from \eqref{centrec}, \eqref{hilb1} that 
\begin{equation}
\label{degrees}
d_k = \frac{1}{N} \sum_{H\in \A} n_H m_H +b_k,
\end{equation}
where $b_k$ are degrees of appearance of the module $\Phi(V)^*$  in the space of 
coinvariants
$S V^*/((SV^*)^W_+)$.  


We are going to show that vector fields $L^{(k)}$ form a free basis for multiarrangement $\A$. We adapt approach for $m=0$ from \cite{T80} (see also \cite{OS}). We will need to know degree of the determinant of $\mathcal Q$ so we firstly find the sum $\sum b_k$ in the next two propositions.

\begin{proposition}\label{bb}
Let $\tau$ be an irreducible representation of $W$ such that $\tau = {\rm  kz }_m(V)$ for some integer multiplicity function $m$. Let $\mathcal C$ be a $W$-orbit on the set of reflection hyperplanes $\mathcal A$. Let $1_{\mathcal C}$ be the characteristic function for $\mathcal C$, that is $1_{\mathcal C} (H)=1$ if $H \in \mathcal C$ and $1_{\mathcal C} (H)=0$ if $H \in \mathcal A\setminus \mathcal C$. Let $\tau' = {\rm kz}_{1_{\mathcal C}} \tau$. Let $b_k = b_k(\tau)$ be degrees of appearance of the dual representation $\tau^*$ in coinvariants $S V^*/((SV^*)^W_+)$, $k=1, \ldots, N$. Let $b_k'=b_k(\tau')$. Then
\beq{propr}
\sum_{k=1}^N b_k =\sum _{k=1}^N b_k'.
\eeq
\end{proposition}
\begin{proof}
Let $\delta = \prod_{H\in \mathcal C} \alpha_H^{-(n_H-1)}$. Let ${\mathcal R}_{\mathcal C}$ be the set of reflections each of which fixes a hyperplane from $\mathcal C$. Let  $|{\mathcal R}_{\mathcal C}|$ be the cardinality of the set  ${\mathcal R}_{\mathcal C}$, and let $|{\mathcal C}|$ be the size of $\mathcal C$. Then 
$$
\deg \delta = -  \sum_{H\in \mathcal C} (n_H-1) = - |{\mathcal R}_{\mathcal C}|.
$$
It is established in \cite[Section 8.3]{BC} (see also \cite{Op}) that
\beq{fake1}
t^{\deg \delta} \chi_{\epsilon \otimes \tau} (t) = t^{c_{\tau'}(-1_{\mathcal C})} \chi_{\tau'}(t),
\eeq
where by e.g. \cite[Corollary 7.18]{BC}
$$
c_{\tau'}(-1_{\mathcal C}) = c_V(-1_{\mathcal C}) = -\frac{1}{N} \sum_{H\in \mathcal C} n_H = - \frac{|{\mathcal R}_{\mathcal C}| + |{\mathcal C}|}{N},
$$
 and $\epsilon $ is the one-dimensional representation of $W$ on the linear space spanned by $\delta$.
Note that character of $\epsilon$ is $\det_{\mathcal C}^{-1}$, where $\det_{\mathcal C}$ is the determinant character for the cyclic subgroup $W_H\subseteq W$ which contains elements of $W$ fixing a chosen hyperplane $H\in  \mathcal C$.

The relation (\ref{fake1}) has the form
\beq{fake2}
t^{-|{\mathcal R}_{\mathcal C}|+\frac{1}{N}(|{\mathcal R}_{\mathcal C}| + |{\mathcal C}|)}   \chi_{\epsilon \otimes \tau} (t) =  \chi_{\tau'}(t),
\eeq
which implies that
\beq{interm0}
(1-N)|{\mathcal R}_{\mathcal C}|+|{\mathcal C}| + \sum_{k=1}^N b_k(\epsilon \otimes \tau) = \sum_{k=1}^N b_k'.
\eeq
Consider  decomposition of the restricted module $\tau^*$  into irreducible ones 
\begin{equation}
    \label{decomp1}
\text{Res}^W_{W_H} \tau^* = \oplus_{j=0}^{n_H-1} r_{H, j} {\det}_{W_H}^{-j},
\end{equation}
where ${\det}_{W_H}$ is the determinant representation for the subgroup $W_H$ and $r_{H, j}\in \N$.
It is shown in \cite[Lemma 2.1]{Op} that 
\begin{equation}
    \label{sum1}
\sum_{k=1}^N b_k (\tau) = \sum_{H\in \mathcal A} \sum_{j=1}^{n_H-1} j r_{H, j}.
\end{equation}
Let us  also consider decomposition of the restricted module $\epsilon^*\otimes\tau^*$ into irreducible ones: 
$$
\text{Res}^W_{W_H} \epsilon^*\otimes\tau^* = \oplus_{j=0}^{n_H-1} \tilde r_{H, j} {\det}_{W_H}^{-j},
$$
where integers $\tilde r_{H, j}$ satisfy 
$$
\tilde r_{H, j}= \begin{cases}
r_{H,j}, \quad H \notin {\mathcal C}, \\
r_{H,j+1}, \quad H \in {\mathcal C}, j\ne n_H-1, \\
r_{H,0},  \quad H \in {\mathcal C}, j = n_H-1.
\end{cases}
$$
It follows that 
$$
\sum_{k=1}^N b_k (\epsilon\otimes \tau) = \sum_{H\in \mathcal A} \sum_{j=1}^{n_H-1} j \tilde r_{H, j} = 
 \sum_{H\in \mathcal A} \sum_{j=1}^{n_H-1} j  r_{H, j}  - \sum_{H \in \mathcal C}\left( \sum_{j=1}^{n_H-1}  r_{H,j}
-(n_H-1) r_{H, 0}\right)
$$
\beq{fin}
=\sum_{k=1}^N b_k (\tau) - |\mathcal C| N + |\mathcal C|  n_H  r_{H, 0}
\eeq
since  $\sum_{j=0}^{n_H-1}  r_{H,j} = \dim \tau^* =N$. It is shown in \cite[Corollary 3.8(vi)]{Op} that numbers $r_{H,j}$ are preserved under KZ twists. Hence $r_{H, 0}=N-1$ as this is easy to see for $\tau = V$. Since $|{\mathcal R}_{\mathcal C}|=|{\mathcal C}| (n_H-1)$ the relation \eqref{propr} follows from equalities \eqref{interm0}, \eqref{fin}.
\end{proof}

Proposition \ref{bb} has the following corollary.

\begin{proposition}\label{bbcor}
Let $\tau$ be an irreducible representation of $W$ such that $\tau = {\rm kz}_m(V)$ for some integer multiplicity function $m$. Then $$\sum_{k=1}^N b_k = |\mathcal A|.$$
\end{proposition}
\begin{proof}
The required statement is well known for $\tau = V$ \cite{OT} (and it also follows from formula  \eqref{sum1} since for $\tau = V$ we have  $r_{H,0}=N-1$, $r_{H,1}=1$, and $r_{H,j}=0$ for $j\ge 2$ in \eqref{decomp1}). Also 
it is established in \cite{BC} that ${\rm kz}_m \circ {\rm kz}_{m'} = {\rm kz}_{m+m'}$ for any integer multiplicity functions $m, m'$.  Hence the statement follows by applying Proposition \ref{bb}. 
\end{proof}

Now we can proceed to find determinant of $\mathcal Q$.
\begin{lemma}\label{propor}
We have $\det {\mathcal Q} = a \prod_{H\in \A} \a_H(x)^{m_H n_H +1}$ for some $a\in \C$.
\end{lemma}

\begin{proof}
It is easy to see by considering linear combinations of rows of the matric  ${\mathcal Q}$ that 
$$
\det {\mathcal Q} = a \prod_{H\in \A} \a_H(x)^{m_H n_H +1}
$$ 
for some $a\in \C[x_1, \ldots, x_N]$. It follows from \eqref{degrees} and Proposition \ref{bbcor} that 
$$
\deg \det {\mathcal Q} = \sum_{k=1}^N d_k = \sum_{H\in \A} m_H n_H + \sum_{k=1}^N b_k = \sum_{H\in \A} m_H n_H + |\A|
$$ 
so  $a$ is constant.
\end{proof}

\begin{lemma}\label{lem2}
In the notation of Lemma \ref{propor} the constant $a \ne 0$.
\end{lemma}
\begin{proof}
Recall that $Q_m$ is a free module over $(SV^*)^W$ of dimension ${\mathcal N}=|W|$. 
One can take polynomials $q_i^{(k)}$, $i,k=1, \ldots, {\mathcal N}$ as part of a free basis $\{f_1, \ldots, f_{{\mathcal N}}\}$ of $Q_m$ over $(SV^*)^W$. 

Let us fix $v \in V$ outside of all the reflection hyperplanes. Consider the linear span  $U$ of vectors $(f_i(g_1v), \ldots, f_{i}(g_{\mathcal N} v))\in \C^{\mathcal N}$ ($i=1, \ldots, {\mathcal N}$), where $g_1, \ldots, g_{\mathcal N}$ is an ordered list of all the elements of $W$. We claim that $U = \C^{\mathcal N}$. Indeed, the linear space $U$ coincides with the space spanned by vectors $(p(g_1v), \ldots, p(g_{\mathcal N}v))$ as $p$ runs over the space of quasi-invariants $Q_m$. It is known that the space $\widetilde U$ spanned by vectors  $(p(g_1v), \ldots, p(g_{\mathcal N}v))$ where $p\in \C[x_1, \ldots, x_N]$ coincides with $\C^{\mathcal N}$. Notice that for any polynomial $p\in \C[x_1, \ldots, x_N]$ the polynomial $q= p \prod_{H\in \A} \a_H (x)^{m_H n_H} \in Q_m$. Since the polynomial $\prod_{H\in \A} \a_H (x)^{m_H n_H}$ is $W$-invariant it follows that $U = \widetilde U = \C^{\mathcal N}$. 

Let us suppose now that $a=0$. Then $\det {\mathcal Q}(v) =0$ so the rows of the matrix ${\mathcal Q}(v)$ are dependent. We assume without loss of generality that 
\begin{equation}
\label{eq1}
q^{(N)}_i(v) = \sum_{k=1}^N b_k q_i^{(k)}(v)
\end{equation}
for some $b_k \in \C$ and all $i=1, \ldots, N$. From $W$-invariance of fields $L^{(k)}$ we get that for any $g\in W$, $i, k=1, \ldots, N$, 
\begin{equation}
\label{eq2}
q_i^{(k)} (g v) = \sum_{j=1}^N a_{ij} q_j^{(k)}(v),
\end{equation}
 where $a_{ij}\in \C$ correspond to the matrix of $g^{-1}\colon V \to V$ and do not depend on $k$. 
It follows from \eqref{eq1}, \eqref{eq2} that
$$
q_i^{(N)}(g v) = \sum_{k=1}^N b_k q_i^{(k)}(gv),
$$
and hence $\dim U < {\mathcal N}$, which is a contradiction.
\end{proof}

To prove the freeness of $D_m$, we introduce a fundamental method to check the freeness of hyperplane arrangements.

\begin{theorem}[Saito's criterion, \cite{Sacriterion}, \cite{Z}]
Let $\A$ be an (affine) arrangement in $V$, and $r\colon \A \rightarrow \N$. Then $\theta_1,\ldots,\theta_N \in D(\A,r)$ form a free basis for 
$D(\A,r)$ if and only if 
$$
\det (\theta_i(x_j))_{1 \le i,j \le N} =c \prod_{H \in \A} \alpha_H^{r(H)}
$$
for some nonzero scalar $c$. Moreover, if $\A$ is central and 
$\theta_1,\ldots,\theta_N \in D(\A,r)$ are homogeneous derivations, then $\theta_1,\ldots,\theta_N$ form a free basis for $D(\A,r)$ if and only if they satisfy the following two conditions:
\begin{itemize}
    \item [(1)] $\theta_1,\ldots,\theta_N$ are independent over $SV^*$.
    \item[(2)] $\sum_{i=1}^N \deg \theta_i=\sum_{H \in \A} r(H)$.
\end{itemize}
\end{theorem}

Now Lemmas \ref{propor}, \ref{lem2} together with Saito's criterion imply the following statement.

\begin{theorem}\label{thmFreeAll}
The space $D_m$ is a free module over $S V^*$ with a basis given by vector fields $L^{(k)}$, $k=1, \ldots, N$. These vector fields are homogeneous of degree $d_k$ given by \eqref{degrees}. 
\end{theorem}

The freeness of $D_m$ was established earlier in the real cases. Namely, the case $m=0$ was established
in \cite{Sa1}, the case $m=const$ was established in 
\cite{T}, and the general invariant multiplicity function $m$ was treated in \cite{ATWaka}.
Furthermore,
in 
\cite{HMRS} the freeness of $D_m$ with constant $m$ was established for well-generated complex reflection groups $W$ and for the infinite series $W=G(r, p, N)$  (with a slightly different formula for the exponents). Theorem \ref{thmFreeAll} includes and extends all these results.

We note that the proof of Theorem \ref{thmFreeAll} is uniform and it makes use of properties of the degrees \eqref{degrees} without having explicit formulas for the degrees $b_k$, which is unusual in theory of hyperplane  arrangements.

Consider now the space of logarithmic polynomial vector fields $\widetilde{D}_m$ for the arrangement $\A$ where the multiplicity of a hyperplane $H\in \A$ is $m_H n_H$. The following statement takes place.




\begin{theorem}\label{thm311} There is an isomorphism 
$\rho\colon \mathbf{Q}_m(V) \xrightarrow{\sim}  \widetilde{D}_m$ of $SV^*$-modules, so the module $\widetilde{D}_m$ is free. The $SV^*$-module $\widetilde{D}_m$ has a free basis consisting of $N$ homogeneous elements $L^{(k)}
=\sum_{i=1}^{N}f^{(k)}_i\partial_i$, where $\deg{f_i^{(k)}}=\dfrac{1}{N}\sum_{H \in \mathcal{A}}m_Hn_H$.
Moreover, the map $\rho$ is an isomorphism of $\C{W}$-modules. Thus the $\C$-span $\langle L^{(k)} \rangle$ in $\widetilde{D}_m$ is isomorphic to 
$\Phi(V)$ as $\C{W}$-modules.
\end{theorem}
\begin{proof}
Note that the mp $\widehat \rho\colon $$SV^* \otimes V \to \Der$ defined by 
$\widehat\rho (\sum_{i=1}^{N}f_i\otimes e_i) = \sum_{i=1}^{N}f_i\partial_i$
is an isomorphism of $SV^*$-modules.
Let $\rho$ be the restriction of the map $\widehat \rho$ to the subspace $\mathbf{Q}_m(V)$. We need to prove that 
$\varphi=\sum_{i=1}^{N}f_i \otimes e_i \in \mathbf{Q}_m(V)$ if and only if 
$L= \widehat\rho(\varphi)=\sum_{i=1}^{N}f_i\partial_i \in \widetilde{D}_m$.

The condition 
$L \in \widetilde{D}_m$ is equivalent to divisibility 
\begin{eqnarray}
\alpha_H^{m_Hn_H} \ | \ \sum_{i=1}^{N}f_i(\partial_i(\alpha_H)) \label{cond1}
\end{eqnarray}
in $SV^*$ 
for any $H \in \mathcal{A}$.
On the other hand, the condition  $\varphi \in \mathbf{Q}_m(V)$ is equivalent to the conditions
\begin{eqnarray}
\sum_{i=1}^{N}f_i \otimes ({e}_{H,j}e_i)
\equiv 0 \mod \langle \alpha_H \rangle^{m_Hn_H} \otimes V\label{cond2}
\end{eqnarray}
for all $1 \leq j \leq n_H-1$ and $H \in \mathcal{A}$.

For a generator $s_H \in W_H$ we have 
$s_He_i=e_i-\alpha_H(e_i)\alpha_H^\vee$ and $s_H(\alpha_H^\vee)=\lambda\alpha_H^\vee$,
where $\lambda$ is a primitive $n_H$-th root of unity. Then we have 
$s_H^ue_i=e_i-\dfrac{1-\lambda^u}{1-\lambda}\alpha_H(e_i)\alpha_H^\vee \ (0 \leq u \leq n_H-1)$.
For $1 \leq j \leq n_H-1$, we get
\[
{e}_{H,j}e_i=\sum_{u=0}^{n_H-1}\lambda^{ju}s_H^ue_i
=\sum_{u=0}^{n_H-1}\dfrac{\lambda^{(j+1)u}}{1-\lambda}\alpha_H(e_i)\alpha_H^\vee,
\]
which is equal to $\dfrac{n_H}{1-\lambda}\alpha_H(e_i)\alpha_H^\vee$ for $j=n_H-1$, 
and it is equal to $0$ for $1 \leq j \leq n_H-2$.
Hence the condition (\ref{cond2}) is equivalent to 
\[
\left(\sum_{i=1}^{N}\alpha_H(e_i)f_i\right)\otimes \alpha_H^\vee \equiv 0
\mod \langle \alpha_H \rangle^{m_Hn_H} \otimes V.
\]
Since $\partial_i(\alpha_H)=\alpha_H(e_i)$, this is equivalent to (\ref{cond1}).

Moreover, the map $\widehat \rho$ is $W$-equivariant for the natural diagonal $W$-action of $SV^* \otimes V$.
Since $w(1\otimes e_{H,i})=w\otimes (we_{H,i})=w \otimes (e_{wH,i}w)=(1 \otimes e_{wH,i})w$ for $w \in W, H \in \mathcal{A}$ and $1 \leq i \leq n_H-1$,
the space $\mathbf{Q}_m(V)$ is $\C{W}$-stable in $SV^* \otimes V$. 
Note that the diagonal $\C{W}$-action of $\mathbf{Q}_m(V)$ extends to the action of $\mathcal{H}_m$ on $\mathbf{Q}_m(V)$.
Since $\mathbf{Q}_m(V)$ is isomorphic to the standard module $M_m(\Phi(V))$ as $\mathcal{H}_m$-modules, we get that the the $\C$-span $\langle L^{(k)}\colon 1 \le k \le N \rangle$ of the generating set $L^{(k)}$
 in $\widetilde{D}_m$ is isomorphic to 
$\Phi(V)$ as $\C{W}$-modules.
\end{proof}


Freeness of $\widetilde D_m$ was established earlier  
in \cite{ST} for real $W$ and $m=1$, and it was established  in \cite{ATWaka} for real $W$ and general invariant multiplicity function $m$. Furthermore, for constant $m$ the freeness was established in \cite{HMRS} for well-generated complex reflection groups $W$ and for the infinite series $W = G(r, p, N)$.

Let us now give a few examples illustrating Theorems \ref{mth1}, \ref{thmFreeAll}, \ref{thm311}.

\begin{example}

Let $V=\C^2$ and $W=G(3,1,2)$. Let multiplicity function $m=1$. Free generators of various modules are given in the following table.
\[
\begin{array}{c||c}
\hline
\text{Basic invariants} & x_1^3+x_2^3, \quad x_1^3x_2^3 \\
\hline
\text{Basis of $Q_1^{V^*}$} & 
x_1^7-7x_1^4x_2^3, \quad  -7x_1^3x_2^4+x_2^7, \quad \ 
 x_1^{10}+5x_1^4x_2^6, \quad 5x_1^6x_2^4+x_2^{10} \ 
 \\
\hline
\text{Basis of ${\bf Q}_1(V)$} & 
x_1^3(x_1^3-4x_2^3)\otimes e_1-3x_1^2x_2^4\otimes e_2,\quad  
3x_1^4x_2^2\otimes e_1+ x_2^3(4x_1^3-x_2^3)\otimes e_2 \\
\hline
\text{Invariant basis of $D_1$} & 
x_1^4(x_1^3-7x_2^3) \partial_1 + x_2^4 (x_2^3 -7x_1^3)\partial_2, \quad  x_1^4(x_1^{6}+5x_2^6) \partial_1 + x_2^4 (5x_1^6+x_2^{6})\partial_2 \ 

\\
\hline
\text{Basis of $\widetilde{D}_1$} &  
x_1^3(x_1^3-4x_2^3)\partial_1-3x_1^2x_2^4\partial_2,\quad  
3x_1^4x_2^2\partial_1+ x_2^3(4x_1^3-x_2^3)\partial_2 \\
\hline
\end{array}
\]
\end{example}

\begin{example}
Let $V=\C^2$, and let $W$ be the dihedral group $I_2(6)$ of order $12$.
There are $6$ reflection hyperplanes: $x_1=0,x_1\pm\sqrt{3}x_2=0$ with multiplicity $m_1$, and $x_2=0$, $\sqrt{3} x_1\pm x_2=0$ with multiplicity $m_2$. For $m_1=m_2=1$, the free generators of the corresponding module $\widetilde{D}_{1,1}$ are 
\begin{eqnarray*}
&{}&L^{(1)}_{1,1} = \left(\dfrac{x_1^6}{5}-2x_1^4x_2^2+x_1^2x_2^4\right)\partial_1+\left(-\dfrac{4}{3}x_1^3x_2^3+\dfrac{4x_1x_2^5}{5}\right)\partial_2, \\
&{}&L^{(2)}_{1,1} =
\left(-\dfrac{3x_1^5x_2}{5}+x_1^3x_2^3\right)\partial_1+\left(-\dfrac{3}{4}x_1^4x_2^2+\dfrac{3x_1^2x_2^4}{2}-\dfrac{3x_2^6}{20}\right)\partial_2.
\end{eqnarray*}
Vector fields $L^{(1)}_{1,1}$, $L^{(2)}_{1,1}$ span the module isomorphic to the reflection representation $V$. 

For $m_1=2,m_1=1$, the free generators of the corresponding module  $\widetilde{D}_{2,1}$ are
\begin{eqnarray*}
&{}&L^{(1)}_{2,1}=\left(\dfrac{x_1^8x_2}{7}+\dfrac{26x_1^6x_2^3}{15}-x_1^4x_2^5\right)\partial_1+\left(
\dfrac{x_1^7x_2^2}{2}+\dfrac{x_1^5x_2^4}{2}-\dfrac{x_1^3x_2^6}{2}-\dfrac{3x_1x_2^8}{70}
\right)\partial_2, \\
&{}&L^{(2)}_{2,1} = \left(
\dfrac{x_1^9}{7}-\dfrac{22x_1^7x_2^2}{7}-x_1^5x_2^4
\right)\partial_1+\left(
-\dfrac{25x_1^6x_2^3}{6}+\dfrac{9x_1^4x_2^5}{2}-\dfrac{51x_1^2x_2^7}{14}+\dfrac{9x_2^9}{14}
\right)\partial_2. 
\end{eqnarray*}
Vector fields $L^{(1)}_{2,1}$, $L^{(2)}_{2,1}$ span the two-dimensional $I_2(6)$-module $\Phi(V)\ncong V$. In this case the KZ twist is non-trivial.\\

\end{example}


\section{K. Saito's primitive derivation and quasi-invariants}
\label{quasi-prim}

In this section we study applications of theory of logarithmic vector fields to 
quasi-invariants by using Theorem \ref{mth1}. 


Reflections $s_H$, $H \in \mathcal A$ can be represented as  $s_H (x) = x - \a_H(x) \alpha_H^\vee$, where $x \in V$ and $\alpha_H^\vee \in V$ is an eigenvector of $s_H$ with eigenvalue different from 1. Consider an arbitrary $m$-quasi-invarint  $p(x)\in Q_m$, and let $\partial_{\a_H^\vee} p(x)$ be its derivative in the direction of $\a_H^\vee$. Suppose that $m_H \ge 1$ for some $H\in \mathcal A$. Then Proposition \ref{propidemp} implies that
\begin{equation}
\label{quasiR1}
\alpha_H(x)^{m_H n_H}| (1-s_H)p(x),
\end{equation}
where we can assume that $s_H$ has order $n_H$.
Therefore
 \begin{equation}
\label{derivquasi}
\a_H(x)^{n_H-1} | \partial_{\alpha_H^\vee} p(x).  
\end{equation}
Indeed, this can be seen by considering Taylor expansion of the polynomial $p$ at a generic point on $H$ in the 
orthonormal coordinate system with respect to a $W$-invariant Hermitian inner product such that one of the coordinate vectors is $\a_H^\vee$ and the corresponding coordinate is proportional to $\a_H(x)$. 
 

Let $x_1, \ldots, x_N$ be coordinates on $V$, and let $y_1, \ldots, y_N\in (SV^*)^W$ be homogeneous basic invariants such that 
$\deg y_1 \le \cdots \le \deg y_N$. 
Let $\mathcal D=\p_{y_N}$ be
K. Saito's primitive derivation, see \cite{Sa1} and \cite{Sa2} for example. Note that different choices of basic invariants lead to proportional primitive derivations unless the highest degree of $W$ is repeated. This can happen for some non-well-generated groups $W$, in which case there is a two-dimension space of vector fields $\mathcal D$.  
Let us recall the following expression for $\mathcal D$ \cite{Y}:
 
\begin{equation*}
{\mathcal D} = J^{-1} \det \mathcal S,
\end{equation*}
where
\begin{equation*}
\mathcal S= 
\left(\begin{matrix}
\frac{\p y_1}{\p x_1} & \ldots &  \frac{\p y_1}{\p x_N} \\
\vdots&\ldots&\vdots\\
\frac{\p y_{N-1}}{\p x_1} & \ldots &  \frac{\p y_{N-1}}{\p x_N} \\
\p_{x_1} & \ldots &  \p_{x_N} \\
\end{matrix}
\right),
\end{equation*} 
and $J=\det (\frac{\p y_i}{\p x_j})_{i,j=1}^N$ is the Jacobian  
matrix. Recall that $J$ is proportional to $\prod_{H\in \mathcal A} \alpha_H(x)^{n_H-1}$ \cite{LT}.

\begin{proposition}\label{quasipreserve}
Let us suppose that $W$-invariant multiplicity function $m$ satisfies $m_H \ge 1$ for all $H\in \mathcal A$. Then $\mathcal D (Q_m) \subseteq Q_{m-1}$.
\end{proposition}
\begin{proof}
Let $q\in Q_m$. Let us show first that $\mathcal D q$ is a polynomial. Let $H \in \mathcal A$. 
Consider the linear combination of columns of the matrix $\mathcal S$ with coefficients $\alpha_1, \ldots, \alpha_N$ such that  $\alpha_H^\vee = (\alpha_1, \ldots, \alpha_N)$. The resulting column after applying the operator to the polynomial $q$  has the form 
$(\p_{\alpha_H^\vee} y_1, \ldots, \p_{\alpha_H^\vee} y_{N-1}, \p_{\alpha_H^\vee} q)^T$. It follows from the property \eqref{derivquasi} that every coordinate of this vector is divisible by $\alpha_H(x)^{n_H-1}$ hence $\mathcal D q$ is regular generically on $H$. Therefore $\mathcal D q$ is a polynomial.

Let us show now that $\mathcal D q \in Q_{m-1}$. Consider $\widetilde q = (1 - s_H) \mathcal D q$. Since $\mathcal D$ is $W$-invariant 
$$
\widetilde q = \mathcal D \left( (1 - s_H)  q\right)  = \mathcal D \left(\alpha_H(x)^{ m_H n_H} r(x)\right)
$$
for some polynomial $r(x)$ as $q\in Q_m$. Since the last row in the matrix 
$$\mathcal S \left(\alpha_H(x)^{  m_H n_H} r(x)\right)$$ 
is divisible by $\alpha_H^{m_H n_H - 1}$  it follows that $\widetilde q \in Q_{m-1}$.
\end{proof}


Let us recall the \textit{basis matrices} for the logarithmic vector fields $D_m$. A
basis matrix $B_m$ for $D_m$ is defined by the property that
$$
[\theta_1,\ldots,\theta_N]^T:=B_m[\partial_{x_1},\ldots,\partial_{x_N}]^T
$$
form a basis for $D_m^W$ over $(SV^*)^W$, and for 
$D_m$ over $SV^*$. In the case of constant multiplicity $m$ Terao gave formulas for $B_m$ in \cite{T} (as well as formulas for the basis matrices for $\widetilde{D}_m$). Basis matrices for complex reflection groups were considered  in \cite{HMRS}. These matrices also give explicit bases for the corresponding quasi-invariants.

\begin{corollary}
Basis matrix gives a basis of $Q_{m}^{V^*}$ over $(SV^*)^W$.  
Namely, for 
$$
[\theta_1,\ldots,\theta_N]^T:=B_m[\partial_{x_1},\ldots,\partial_{x_N}]^T,
$$
the polynomials $\theta_i(x_j)\ (i,j=1,\ldots,N)$ form a basis for 
$Q_m^{V^*}$ over $(SV^*)^W$. 
%
\end{corollary}

\begin{proof}
Since $Q_m^{V^*} \cong D_m^W \otimes V^*$ by Theorem \ref{mth1}, the statement follows immediately.
\end{proof}

Now let us consider the relation between $Q_m^{V^*}$ and 
$Q_{m-1}^{V^*}$ for a constant $m$. To do that, the key is K. Saito's primitive 
derivation $\mathcal D$.
Let $T:=\C[y_1,\ldots,y_{N-1}]$,
$\text{Der}_{(0)}$ be the localization of $\text{Der}$ at the ideal $(0)$ of polynomials vanishing at 0, and let 
$\nabla_{\mathcal{D}}:\text{Der}_{(0)} \rightarrow \text{Der}_{(0)}$ be the affine connection 
defined by 
$$
\nabla_{\mathcal{D}}(\sum_{i=1}^N f_i \partial_{x_i}):=
\sum_{i=1}^N \mathcal{D}(f_i) \partial_{x_i}.
$$
In \cite{Sa1}, \cite{Sa2} and \cite{ARSY}, it is shown that 
$
\nabla_{\mathcal{D}}$ gives a bijection 
$$
\nabla_{\mathcal{D}}\colon \text{Der}^W \rightarrow \text{Der}_{(SV^*)^W}.
$$
Here $\text{Der}_{(SV^*)^W}\subseteq  \text{Der}_{(0)}$ is the derivation module of the invariant subring 
$(SV^*)^W$. 
Thus for $\theta \in \text{Der}^W \subseteq  \text{Der}_{(SV^*)^W}$, we can define 
$\nabla^{-1}_{\mathcal{D}} \theta \in \text{Der}^W$.
By using this we can show the following.

\begin{theorem}[\cite{Sa1}, \cite{T05}, \cite{AT}, \cite{ARSY}]
Suppose that the group $W$ is well-generated and the multiplicty function $m$ is constant. Let $k$ be a non-negative integer. Then 
$\nabla_{\mathcal{D}}$ is a $T$-isomorphism from $D_{k+1}^W$ onto $D_{k}^W$.
Moreover, 
there exists a basis $\theta_1,\ldots,\theta_N$ for $D_0^W$ such that, for 
$\eta_i^{(k)}:=\nabla_{\mathcal{D}}^{-k} \theta_i \in D_k^W$, it holds that 
\begin{itemize}
\item[(1)]
$\eta_1^{(k)},\ldots,\eta_N^{(k)}$ are independent over $(S V^*)^W$, and 
form a basis for $D_k^W$, 
\item[(2)]
$\{\eta_i^{(k)}\}_{1 \le i \le N,\ k \ge 0}$ are independent over $T$, 
\item[(3)]
Let $\mathcal{F}_k:=\bigoplus_{i=1}^N T \cdot \eta_i^{(k)}$. Then 
$$
D_m^W=\bigoplus_{k \ge m} \mathcal{F}_k.
$$
Moreover, $\nabla_{\mathcal{D}}$ induces a $T$-isomorphism $\nabla_{\mathcal{D}}\colon 
\mathcal{F}_{k} \rightarrow \mathcal{F}_{k-1}$ for $ k\ge 1$, which 
induces 
the $T$-isomorphism 
$$
\nabla_{\mathcal{D}}\colon D_m^W \rightarrow D_{m-1}^W.
$$
\end{itemize}
\label{Hodge}
\end{theorem}

\begin{corollary}\label{corDisom-wellgen}
Suppose that the group $W$ is 
well-generated and the multiplicity function $m$ is a non-zero constant. Then  
${\mathcal D}\colon  Q_m^{V^*} \to Q_{m-1}^{V^*}$ is an isomorphism of vector spaces. 
\end{corollary}

\begin{proof}
By Theorems \ref{mth1} and \ref{Hodge}, 
\begin{eqnarray*}
Q_m^{V^*} &\cong& D_m^W \otimes V^*\\
&\cong& (\bigoplus_{k \ge m} \mathcal{F}_k ) \otimes V^*\\
&\cong& \bigoplus_{k \ge m} (\mathcal{F}_k \otimes V^*)\\
&\cong& \bigoplus_{k \ge m} \bigoplus_{i,j=1}^N T \cdot \eta_{i}^{k}(x_j).
\end{eqnarray*}
Since $\mathcal{D}$ is $T$-linear, and $\mathcal{D}(\eta_i^{(k)}(x_j))=
\eta_i^{(k-1)}(x_j)$ by Theorem \ref{Hodge},  
we complete the proof. 
\end{proof}
Finally, we return to the case of general non-constant invariant multiplicity function for Coxeter groups.

\begin{theorem}
Assume that $W$ is an (irreducible) Coxeter group with multiplicity function $m\ge 1$.  Then 
${\mathcal D}\colon Q_m^{V^*} \to Q_{m-1}^{V^*}$ is a $T$-isomorphism.  
\end{theorem}
\begin{proof}
The cases of $W$ of type $B_N$ and $F_4$ follow from Theorem \ref{mth1} and \cite{ATWaka}. By Corollary~\ref{corDisom-wellgen} it is enough to consider the dihedral case $W$ of type $I_2(2\ell)$, $\ell \ge 2$ with the values of multiplicity function  $m_1 > m_2 \ge 1$. 
Let $|m|=m_1+m_2$. 
Let us assume that the dihedral configuration has the defining equation $z^{2 \ell} = \bar z^{2 \ell}$ in complex coordinates $z, \bar z$. 

The basic invariants then have the form $y_1 = z \bar z$, $y_2 = \frac{1}{2\ell}(z^{2 \ell}+ \bar z^{2 \ell})$, and the Saito primitive derivative is 
\begin{equation}
\label{di1}
\mathcal D = \partial_{y_2} = \frac{1}{z^{2\ell}-\bar z^{2 \ell}} (z \partial_z - \bar z \partial_{\bar z}).
\end{equation}
It is shown in \cite{Fdih} that there exist quasi-invariants $q_i^{(m)}, \, p_i^{(m)} \in Q_m$ of the form 
\begin{equation}
\label{di2}
q_i^{(m)} = \sum_{s=0}^{|m|} a_s z^{(|m|-s)\ell + i} \bar z^{\ell s}, \quad a_0=1, \quad 
 p_i^{(m)} = \bar q_i^{(m)}, \quad 1\le i \le 2\ell -1, i \ne \ell,
\end{equation}
where $a_s \in \C$. 
Moreover, such quasi-invariants are unique, that is 
if there is $f \in Q_m$ such that 
$f=\sum_{s=0}^{|m|} b_s z^{(|m|-s)\ell + i} \bar z^{\ell s}$ for $b_s \in \C,\ b_0 =1$, then 
$f=q_i^{(m)}$, and the uniqueness of $p_i^{(m)}$ is analogous. 
Also, these quasi-invariants form a basis over $\C[y_1, y_2]$ of the isotypic component of all the two-dimensional irreducible representations \cite{Fdih}. It is easy to see that in the case of even $|m|$ polynomials \eqref{di2} give a basis   of the isotypic component of $Q_m^V = Q_m^{V^*}$  for $i=1, 2\ell-1$, and in the case of odd $|m|$ the basis is given when $i = \ell+1, \ell-1$. Let us assume that $i$ takes these values, 
that is $i \in I$, where $I=\{1, 2 \ell -1\}$ if $|m|$ is odd and $I=\{\ell+1, \ell -1\}$ if $|m|$ is even.

Since $\mathcal D(Q_m^V)\subseteq Q_{m-1}^V$ by Proposition \ref{quasipreserve}, it follows from formulas \eqref{di1}, \eqref{di2}  that
$$
{\mathcal D}(q_i^{(m)}) = (|m|\ell+i) q_i^{(m-1)}, \quad {\mathcal D}(p_i^{(m)}) = (|m|\ell+i) p_i^{(m-1)}.
$$
Therefore ${\mathcal D}(Q_m^V)$ contains the linear subspace $U\subseteq Q_{m-1}^V$ formed by all the linear combinations $\sum_{i \in I} (\sigma_i q_i^{(m-1)} + \tau_i p_i^{(m-1)})$, where $\sigma_i, \tau_i \in \C[y_1]$. 

Note that uniqueness of quasi-invariants $q_i^{(m-1)}$ of the form \eqref{di2} and the property that $q \in Q_m$ if and only if $z\bar z q \in Q_m$ implies that
$$
q_i^{(m)} = 2 \ell y_2 q_i^{(m-1)}+ c_1 (z \bar z)^\ell q_i^{(m-1)} + c_2 (z \bar z)^i p_{2\ell -i}^{(m-1)}
$$
for suitable $c_1, c_2 \in \C$. Hence 
$
q_i^{(m)} -2\ell y_2 q_i^{(m-1)},
p_i^{(m)} - 2 \ell y_2 p_i^{(m-1)}
\in U$.
Therefore
for any $n\in \N$
$$
{\mathcal D}(y_2^n q_i^{(m)}) = (2 n \ell+|m|\ell +i) y_2^{n} q_i^{(m-1)} +u,\quad 
{\mathcal D}(y_2^n p_i^{(m)}) = (2 n \ell +|m|\ell +i) y_2^{n} p_i^{(m-1)} +v, 
$$
where $u, v \in y_2^{n-1} U$.
It follows by induction in $n$ that the map $\mathcal D$ is surjective.

Note that $\mathcal D$ maps the linear span of $q_i^{(m)}$ and $p_i^{(m)}$ over $\C[y_1, y_2]$, $i \in I$,  to the linear span of $q_i^{(m-1)}$ and $p_i^{(m-1)}$ over $\C[y_1, y_2]$. Since $\mathcal D$ is homogeneous and the dimensions of the  corresponding graded components of these vector spaces are equal we get that $\mathcal D$ is an isomorphism.
\end{proof}
\section{Affine version}
\label{affine}

Let now $W$ be an irreducible finite Weyl group and $V\cong \C^N$ its reflection representation. Let $\mathcal R$ be a reduced root system with the Weyl group $W$.
Denote by $(\cdot, \cdot)$ the standard 
bilinear scalar product in $V$, so that reflections $s_\a$, $\a \in \mathcal R$ act by orthogonal transformations $s_\a (x) = x - (\a, x) \alpha^\vee$, where $x \in V$ and $\alpha^\vee = \frac{2 \alpha}{(\alpha, \alpha)}$. Let $m\colon {\mathcal R} \to \N$ be $W$-invariant multiplicity, and denote $m_\alpha = m(\alpha)$.

Consider the following ring $Q_m^{tr}$ of quasi-invariants defined by difference conditions. These rings appear as rings of quantum integrals of trigonometric Calogero--Moser systems \cite{CV}. A polynomial $p(x)\in Q_m^{tr}$ if and only if for any $\alpha \in \mathcal R$ one has
\begin{equation}
\label{trquasi}
p(x+\frac12 j \alpha^\vee) = p(x- \frac12 j \alpha^\vee)
\end{equation}
for any $x \in V$ such that $(\alpha, x)=0$ and $j = 1, \ldots, m_\alpha$. Note that $Q_m^{tr}$ is an $(SV^*)^W\otimes W$-module.

Consider also the Catalan arrangement 
$$
Cat =
\{\Pi_{\alpha, j}\colon \alpha\in {\mathcal R}_+, -m_\alpha\le j \le m_\alpha\},
$$
where 
$
\Pi_{\alpha, j}=\{x\in V\colon (\alpha,x) = j\}
$ and ${\mathcal R}_+$ is a positive half of the root system $\mathcal R$.
Let $D(Cat)$ be the corresponding module of logarithmic vector fields:
$$
D(Cat) = \{L \in \Der\colon ((\alpha,x) -j)| L((\alpha,x)),  \text { for any } \,  \alpha \in {\mathcal R}, -m_\alpha\le j \le m_\alpha\}.
$$

\begin{proposition}
Let $L = \sum_{i=1}^N a_i\partial_i \in \Der^W$. Then $L \in D(Cat) $ if and only if $a_i \in Q_m^{tr}$ for all   $i$.
\end{proposition}
\begin{proof}
Let $L\in D(Cat)$. Then $\sum \alpha_i a_i=0$ if $(\alpha, x)= j$, $1\le j \le m_\alpha$;  $\alpha= (\alpha_1, \ldots, \alpha_N)$. Hence for such $x$ we have $(1-s_\alpha)(a)=0$, where $a=(a_1, \ldots, a_N)$ and $s_\alpha$ acts  on vector $a$ for fixed $x$. Let us denote the resulting vector as $b=(1-s_\alpha)(a)$, $b=(b_1, \ldots, b_N)$. Since $L$ is $W$-invariant we get that for all $i$ the component 
\begin{equation}
\label{interm}
b_i(x)=a_i(x) - a_i(s_\alpha(x))=0
\end{equation}
if $(\alpha, x)=j$. 
Let us represent $x$ as $x=\tilde x + \frac12 j \alpha^\vee$, then $(\alpha, \tilde x)=0$.
Note that  $s_\alpha(x) =\tilde x - \frac12 j \alpha^\vee$. Hence the relation \eqref{interm} implies the quasi-invariant condition \eqref{trquasi}. 

The converse statement follows using the same arguments. Note also that the condition $\sum \alpha_i a_i=0$ if $(\alpha, x)=0$, $\alpha \in \mathcal R$, is always satisfied for a $W$-invariant vector field $L$.
\end{proof}


Let $Q_m^{tr, V}$ denote the isotypic component of $V\cong V^*$ of $W$-module $Q_m^{tr}$. Previous considerations allow to express the module $D(Cat)^W$ as follows.
\begin{theorem}
\label{mth2}
$D(Cat)^W \cong\text{\rm Hom}_W(V,Q_m^{tr, V})$, and
$Q_m^{tr, V}\cong D(Cat)^W\otimes V$ as modules over $(S V^*)^W\otimes \C W$. 
\end{theorem}

Let 
$$
cCat =
\{\Pi_{\alpha, j}^z: \alpha\in {\mathcal R}_+, -m_\alpha\le j \le m_\alpha\}\cup \Pi^z,
$$
where 
$
\Pi_{\alpha, j}^z=\{(x,z)\in V\oplus \C: (\alpha,x) = j z\}
$
and
$
\Pi^z =\{(x,0)\in V\oplus \C: x\in V\}.
$
The freeness of $cCat$ is known by \cite{YC at} and \cite{AT2}. Let us derive the freeness of $D(cCat)$ by using the freeness of quasi-invariant module $Q_m^{tr,V}$.

For a filtered algebra $A$ let $\bar A = gr A$ be its associated graded algebra, and let $gr\colon A \to \bar A$ be the corresponding map. 
The algebra $Q_m^{tr}$ has a natural filtration by the degrees of the polynomials. 
The associated graded algebra  $gr(Q_m^{tr})$
may be thought of as an algebra of the highest order terms of the polynomials from $Q_m^{tr}$. It is stated in \cite[Proposition 6.5]{ERF} that $gr(Q_m^{tr})\cong Q_m$, and that $Q_m^{tr}$ is a free module over the invariant ring $(S V^*)^W$. 

\begin{theorem}
In the above notation, 
\begin{itemize}
\item [(1)]
$D(Cat)^W$ is a free module over $(SV^*)^W$, and it has a basis 
$\theta_1,\ldots,\theta_N$ such that 
the degrees of the leading terms of each $\theta_i$ are equal to 
the degree of an element of a homogeneous basis for $D_m$.

\item [(2)]
$D(cCat)$ is free over $SV^*$.
\end{itemize}
\label{Catfree}
\end{theorem}

\begin{proof}

(1)\,\,
    Consider filtration $Q_m^{tr}=\cup_{i=0}^\infty F_i$, where the space $F_i$ consists of the quasi-invariants of degree less or equal than $i$. There exist complimentary $W$-modules $G_i$ such that $F_i=F_{i-1}\oplus G_i$, where we put $F_{-1}=0$. Let $F_i^V$ and $G_i^V$ be the corresponding isotypic components. Then we have $F_i^V = F_{i-1}^V\oplus G_i^V$.

  For the associated graded algebras we have
  $$
  gr(Q_m^{tr}) \cong Q_m \cong \oplus_{i=0}^\infty F_{i}/F_{i-1} \cong \oplus_{i=0}^\infty G_i,
  $$
 and
   $$
  gr(Q_m^{tr, V})  \cong \oplus_{i=0}^\infty F_{i}^V/F_{i-1}^V \cong \oplus_{i=0}^\infty G_i^V \cong Q_m^V
  $$
  as $(SV^*)^W\otimes W$-modules.
  It follows that $(SV^*)^W$-module $Q_m^{tr, V}$ is free which implies the freness of $D(Cat)^W$ by Theorem \ref{mth2}.
  Moreover, one can choose a basis $\theta_1, \ldots, \theta_N$ in $D(Cat)^W\cong \text{\rm Hom}_W(V, Q_m^{tr, V})$ which is mapped to the elements in the spaces $\text{\rm Hom}_W(V, G_i^V)$ under the associated graded map $gr\colon Q_m^{tr, V}\to Q_m^V$.
  
(2)\,\,
Note that $cCat$ is a central arrangement in $\C^{N+1}$.
 For $\theta=\sum_{i=1}^N f_i(x_1,\ldots,x_N)\partial_{x_i} \in D(Cat)$, let 
\begin{equation}
    \label{thetaprime}
\theta':=z^{\deg \theta}\sum_{i=1}^N f_i(\frac{x_1}{z},\ldots,\frac{x_N}{z})\partial_{x_i},
\end{equation}
where $\deg \theta = \max_i\deg f_i$. It is easy to show that $\theta' \in D(cCat)$. 
Let $\theta_1,\ldots,\theta_N$ be a basis for $D(Cat)^W$ constructed in (1) such that 
$$
\sum_{i=1}^N \deg \theta_i=\sum_{i=1}^N \deg \theta_i'=|Cat|.
$$
By definition vector fields $\theta_1,\ldots,\theta_N$ are independent over $(SV^*)^W$. Since 
$ gr(\theta_1),\ldots,gr(\theta_N) \in D_m$ are independent over $SV^*$ as shown in (1) and 
Theorem \ref{thmFreeAll}, it holds that $\theta_1,\ldots,\theta_N$ are independent over $SV^*$ too. 
Then the corresponding derivations $\theta_1',
\ldots,\theta_N'$ 
are independent over 
$SV^*[z]$. 
Hence the derivations 
$\theta_1',\ldots,\theta_N'$ together with $\theta_E:=
\sum_{i=1}^N x_i\partial_{x_i}+z\partial_z$ form a free basis for 
$D(cCat)$ by Saito's criterion. 
\end{proof}

\begin{remark}
In general the freeness of the deconing of a central arrangement $\mathcal{A}$ does not imply the freeness of $\mathcal{A}$. For example, see a non-free arrangement in $\C^3$ given by  $x_1 x_2 z(x_1 + x_2 + z)=0$ and 
its free deconing in $\C^2$ given by $x_1 x_2 (x_1 + x_2 + 1)=0$ which has a free basis 
$$
\theta_1 = x_1(x_1 + 1)\partial_{x_1}+ x_1 x_2 \partial_{x_2},\quad \theta_2 = x_1 x_2 \partial_{x_1}+ x_2 (x_2+1)\partial_{x_2}.
$$
Note that in this case components of vector fields $\theta_1', \theta_2', \theta_E$ make a matrix with the determinant $x_1 x_2 z^2 (x_1 + x_2 +z)$.
\end{remark}

%
We can define a surjective map 
$\Phi\colon D(Cat)^W \rightarrow D_m^W$ as the map 
which keeps the highest order term.
Let also $gr\colon  Q_m^{tr, V} \to Q_m^V$ be the map which keeps the highest order term. Let $\chi_\delta$, $\delta\in V^*$, map a logarithmic vector field $L$ to the polynomial $\delta(L)$.
 It induces maps $\chi_{\delta, 1}: D(Cat)^W \to  Q_m^{tr, V}$, $\chi_{\delta, 2}:D_m^W \to  Q_m^{V}$ by Theorems \ref{mth2}, \ref{mth1}.

 Previous discussion leads to the following statement.
 
 \begin{proposition}
We have the following commutative diagram:
\begin{equation*}
\begin{tikzcd}
D(Cat)^W \arrow[r,"\chi_{\delta,1}"] \arrow[d,swap, "\Phi"]&
 Q_m^{tr, V} \arrow[d,swap, "gr"]  \\
D_m^W     \arrow[r,"\chi_{\delta,2}"] & Q_m^V 
\end{tikzcd}
\end{equation*}
 \end{proposition}

\subsection{Catalan arrangements of type $BC_N$}
From now on let us consider the non-reduced root system ${\mathcal R}=BC_N$. Its positive half consists of the vectors $e_i, 2 e_i$ ($1\le i \le N$), and vectors $e_i \pm e_j$ ($1\le i<j \le N$), where $e_1, \ldots, e_N$ is the standard basis in $V$. The corresponding finite reflection group $W$ is of type $B_N$. Let us consider the $W$-invariant multiplicity function $m$ which has values $m_1=m(e_i)$, $m_2 = m(2 e_i)$, $m_3 = m(e_i \pm e_j)$. Consider the space of quasi-invariants $Q^{tr}_m(BC_N)$ which in this case by definition consists of polynomials $p(x)$
satisfying conditions
\begin{align}
p(x+s e_j) &= p(x- s e_j) \mbox{ at } x_j =0, s=1, \ldots, m_1,  \label{quasi1}\\
p(x+(s-\frac12)  e_j) &= p(x- (s-\frac12) e_j) \mbox{ at } x_j =0, s=1, \ldots, m_2,\label{quasi2}\\
p(x \pm s e_i) &= p(x + s e_j)  \mbox{ at } x_j =\pm x_i, s=1, \ldots, m_3. \label{quasi3}
\end{align}

Note that $Q_m^{tr}$ is an $(SV^*)^W\otimes W$-module. 
Let us explain the freeness of $Q^{tr}_m(BC_N)$ as $(S V^*)^W \cong \C[x]^W$-module following the outline in \cite{ERF}.

The conditions \eqref{quasi1} -- \eqref{quasi3} can be rearranged as follows. For a vector $\alpha \in V$ define the operator $\delta_\alpha$ which maps a function $f(x)$ to the function $\delta_\alpha f(x) := f(x+\alpha) - f(x - \alpha)$.
\begin{proposition}\label{prop1delta} (\cite{CSV}, \cite{FVr}, \cite{Vrabec}).
Let $l, r \in \N\cup\{0\}$. A polynomial $p(x)$ satisfies conditions
\begin{align}
&p(x+ s\alpha) = p(x - s \alpha) \mbox{ at } (\alpha, x) =0, s=1, \ldots, l, \label{morequasi1}\\
\nonumber \mbox{ and } \\
&p(x+ (l+2 s)\alpha) = p(x - (l + 2s) \alpha) \mbox{ at } (\alpha, x) =0, s=1, \ldots, r \label{morequasi2}
\end{align}
if and only if
\begin{align*}
&(\delta_\alpha \circ (\alpha, x)^{-1})^{s-1} \circ \delta_\alpha p(x) = 0 \mbox{ at } (\alpha, x)=0, s= 1, \ldots, l,\\
\nonumber \mbox{ and } \\
&(\delta_{2 \alpha} \circ (\alpha, x)^{-1})^{t} \circ (\delta_\alpha \circ (\alpha, x)^{-1})^{l-1} \circ \delta_\alpha p(x) = 0 \mbox{ at } (\alpha, x)=0, t= 1, \ldots, r.
\end{align*}
\end{proposition}
This statement has the following corollary.
\begin{proposition}\label{prop2delta}
Suppose a polynomial $p(x)$ satisfies conditions \eqref{morequasi1}, \eqref{morequasi2}. 
Then the highest order term $p_0$ of the polynomial $p$ satisfies
$$
\partial_\alpha^{2 s-1} p_0(x) = 0 \mbox{ at } (\alpha, x)=0, s = 1, \ldots,  l+r.
$$
\end{proposition}
\begin{proof}
By Proposition \ref{prop1delta} we get that the highest term $p_0$ satisfies
$$
(\partial_\alpha \circ (\alpha, x)^{-1})^{s-1} \circ \partial_\alpha p_0(x) = 0 \mbox{ at } (\alpha, x)=0, s= 1, \ldots, l+r,
$$
which implies the statement.
\end{proof}

Consider the space of quasi-invariants $Q_{\widetilde m}$ for the root system $B_N$ with the multiplicity function $\widetilde m( e_i) = m_1+m_2$, $\widetilde m(e_i \pm e_j) = m_3$. 
Proposition \ref{prop2delta} implies the following statement. 
\begin{proposition} \label{prop3delta} 
Let $p \in Q_m^{tr}(BC_N)$. Then the highest order term $p_0\in Q_{\widetilde m}$.  
\end{proposition}

Proposition~\ref{prop3delta} gives the inclusion ${\rm gr} (Q^{tr}_m(BC_N)) \subseteq Q_{\widetilde m}$, and we aim to establish the equality. 
\begin{theorem}\label{BCthm} (cf. \cite{ERF})
We have ${\rm gr} (Q^{tr}_m(BC_N)) = Q_{\widetilde m}$.
\end{theorem}
\begin{proof}
Let us define following \cite{Chalykh} the difference operator
\begin{equation}
\label{MRop}
H = \sum_{j=1}^N a_j^+(T_j-1) +  \sum_{j=1}^N a_j^-(T_j^{-1}-1), 
\end{equation}
where $T_j f(x) = f(x+ e_j)$, and coefficients 
$$
a_j^{\pm} = (1 \mp \frac{m_1 }{x_j})(1 \mp \frac{m_2 }{x_j \pm \frac12})\prod_{\substack{i=1 \\ {i \ne j}}}^N (1\mp \frac{m_3 }{x_j+x_i}) (1\mp \frac{m_3}{x_j - x_i}),
$$
(there seems to be a small typo in the coefficient $a_j^\pm$ in \cite{Chalykh}, see also \cite{Vrabec}).
For any homogeneous polynomial $\sigma \in \C[x]^W$ we define the  
 difference operator $H_\sigma$ given by the formula
$$
H_\sigma = ad^{\deg \sigma}_H \sigma(x), \quad ad_A B = A B - B A.
$$
Note that 
$$
 H =  {\mathcal L} + \widetilde{\mathcal L},
$$
where
$$ {\mathcal L} = \sum_{i=1}^N \frac{\partial^2}{\partial x_i^2} - \sum_{i=1}^N \frac{2(m_1+m_2)}{x_i} \partial_{x_i} - \sum_{i<j}^N \sum_{\epsilon\in\{\pm1\}} \frac{2 m_3}{x_i-\epsilon x_j}(\partial_{x_i}-\epsilon \partial_{x_j})
$$
and $\widetilde{\mathcal L}$ is an infinite sum of differential operators of degrees less than $\deg \widetilde{\mathcal L} = -2$.
Hence 
$$
H_\sigma = L_\sigma + \widetilde{L_\sigma},
$$
 where
$$
 L_\sigma = ad_{\mathcal L}^{\deg \sigma} \sigma(k)
$$
and  $\widetilde{L_\sigma}$ is a sum of differential operators of degrees strictly less than $\deg {L_\sigma}$.

The operators $L_\sigma\colon \C[x]^W\to \C[x]^W$ and polynomials $\sigma$ as multiplication operators on $\C[x]^W$ generate a faithful representation of the spherical rational Cherednik algebra $e {\mathcal H}_{\widetilde m} e$ of type $B_N$ (\cite{EGsympl}, \cite{Heckman}). It is shown in \cite{Chalykh} that the operator \eqref{MRop} preserves the space of quasi-invariants $Q^{tr}_m(BC_N)$. It follows that the algebra $e{\mathcal H}_{\widetilde m}e$ acts on the space ${\rm gr}(Q^{tr}_{m}(BC_N))$, which is a submodule in $Q_{\widetilde m}$, and 
$Q_{\widetilde m}/{\rm gr}(Q^{tr}_{m}(BC_N))$ is semisimple by general results from \cite{BEG}, \cite{GGOR}, thus free over 
$\C[x]^W$. 
Now let 
$$
\Delta:=\prod_{i=1}^N \left(\prod_{s=1}^{m_1}(x_i^2 -s^2) \prod_{t=1}^{m_2} (x_i^2 - (t -\frac12)^2)\right)
\prod_{i<j}^N \left(
\prod_{r=1}^{m_3} \prod_{\epsilon \in \{\pm 1\}} ((x_i-\epsilon x_j)^2 -  r^2)\right)  \in Q_m^{tr}(BC_N)
$$
and $\Delta_0$ be the leading term of $\Delta$. Then by definition it is clear that 
$$
\Delta p(x) \in Q^{tr}_{m}(BC_N),\ 
\Delta_0 p(x) \in Q_{\widetilde m}
$$
for any $p(x) \in \C[x]$. So we have a sequence of $\C[x]^W$-free modules
$$
{\rm gr}(\C[x]\Delta )=\C[x] \Delta_0 
\subseteq {\rm gr}(Q^{tr}_{m}(BC_N)) \subseteq Q_{\widetilde m}.
$$
Since 
$$
\mbox{rank}_{\C[x]^W} \C[x]\Delta_0 =
\mbox{rank}_{\C[x]^W}Q_{\widetilde m}=|W|,
$$
it holds that $\mbox{rank}_{\C[x]^W} {\rm gr}(Q^{tr}_{m}(BC_N))=|W|$. 
Hence
the free $\C[x]^W$-module 
$Q_{\widetilde m}/{\rm gr}(Q^{tr}_{m}(BC_N))$
is of rank zero, implying that 
$Q_{\widetilde m}/{\rm gr}(Q^{tr}_{m}(BC_N))=0$ (cf. close arguments in \cite{ERF}).
\end{proof}

Theorem \ref{BCthm} has the following implication.

\begin{corollary}\label{corBC}(cf. \cite{ERF})
The algebra $Q_m^{tr}(BC_N)$ is a free module over $B_N$-invariants $\C[x]^W$.
\end{corollary}
 
Now let us consider the extended Catalan arrangement $BCCat$ corresponding to the root system ${\mathcal R}= BC_N$. By definition it has defining equation $P=0$, where polynomial $P$ has the form
$$
P(x)=\prod_{i=1}^N \left(x_i \prod_{j=1}^{m_1} (x_i^2 - j^2) \prod_{j=1}^{m_2} (4 x_i^2 - (2 j-1)^2) \right) 
\prod_{i<j}^N 
\prod_{\epsilon \in \{\pm 1\}}
\left((x_i - \epsilon x_j) \prod_{k=1}^{m_3} ((x_i- \epsilon x_j)^2 -  k^2)\right). 
$$
Similarly to Theorem \ref{mth2} we have the following statement for the isotypic component $Q_m^{tr, V}(BC_N) \subseteq Q_m^{tr}(BC_N)$ of $B_N$-module $V$. 
\begin{theorem}
\label{mth3}
$Q_m^{tr, V}(BC_N)\cong D(BCCat)^W\otimes V$ as modules over $(S V^*)^W\otimes \C W$. Moreover, $\text{\rm Hom}_W(V,Q_m^{tr,V}(BC_N)) \cong 
D(BCCat)^W$.
\end{theorem}
Previous considerations also lead us to the following freeness result.
\begin{theorem}\label{BCth9}
The module $D(BCCat)^W$ is free over invariants $(SV^*)^W$. The module $D(BCCat)$ is free over polynomials $SV^*$.
\end{theorem}
\begin{proof}
The first claim follows from Corollary \ref{corBC} and Theorem \ref{mth3}. 
Moreover, it follows from Theorem \ref{BCthm} 
that there exists a free basis $\theta_1, \ldots, \theta_N$ of $(SV^*)^W$--module $D(BCCat)^W$ such that each component of $\theta_i$ has the same degree, and vector fields $\tilde \theta_i$ obtained by taking the highest order terms in the components of $\theta_i$ give an invariant free basis of $SV^*$--module $D_{\widetilde m}$ for the $B_N$ arrangement.
For such a basis $\theta_i$ we get that determinant of the matrix $M$ of coefficients $M =(\theta_i(x_j))_{i,j =1}^N$, is a non-zero scalar  multiple of $P(x)$.
Indeed, the highest order term of $\det M$ is given by replacing $\theta_i(x_j)$ with $\tilde \theta_i(x_j)$ in $M$, and this determinant is non-zero. We also have
$$
\deg \det M = \sum_{i=1}^N \deg \theta_i = \sum_{i=1}^N \deg \tilde\theta_i  = \sum_{\Pi_\alpha \in \mathcal A} (2 \widetilde m(\alpha) +1) = \deg P, 
$$
where $\mathcal A$ is the $B_N$ arrangement, and $\deg \theta_i, \deg \tilde \theta_i$ denote the degrees of the components of the corresponding vector field. Hence $\det M$ is a scalar multiple of $P$. 
It follows that the arrangement $BCCat$ is free by (affine) Saito's  criterion.
\end{proof}

Finally, let us consider the coning $cBCCat$ which is an arrangement in $\C^{N+1}$ with defining polynomial
\begin{multline*}
P_c(x,z)=z \prod_{i=1}^N \left(x_i \prod_{j=1}^{m_1} (x_i^2 - j^2 z^2) \prod_{j=1}^{m_2} (4 x_i^2 - (2 j-1)^2 z^2) \right) \times \\
\prod_{i<j}^N 
\prod_{\epsilon \in \{\pm 1\}}
\left((x_i - \epsilon x_j) \prod_{k=1}^{m_3} ((x_i-\epsilon x_j)^2 -  k^2 z^2)\right), 
\end{multline*}
where $z \in \R$. 
\begin{theorem} \label{BCth10}
The arrangement $cBCCat$ is free.
\end{theorem}
\begin{proof}
Let us consider the basis $\theta_1, \ldots, \theta_N$ of $D(BCCat)$ as in the proof of Theorem \ref{BCth9}. It is easy to see that vector fields $\theta_1', \ldots, \theta_N'$ defined by \eqref{thetaprime} belong to the module of logarithmic vector fields $D(cBCCat)$. Determinant $\det M'$ of the corresponding matrix of coefficients $M'$, $M'_{ij} = \theta_i'(x_j)$,   is a nonzero scalar multiple of $P_c(x, z)/z$.
Note that the sum of degrees of (coefficients of) vector fields
 $\theta_1', \ldots, \theta_N', \theta_E=\sum_{i=1}^N x_i \partial_{x_i} + z \partial_z$ 
 is equal to the degree of the polynomial $P_c(x, z)$
 since 
 the sum of degrees 
 $$
 \sum_{i=1}^N \deg\theta_i'=  \sum_{i=1}^N \deg\theta_i'|_{z=1}
= \sum_{i=1}^N \deg\theta_i=
 \deg P(x) = \deg P_c(x,z)-1.
 $$
 Hence
 these vector fields form a free basis of $D(cBCCat)$ over $\C[x_1, \ldots, x_N, z]$ by Saito's criterion. 
\end{proof}

The proofs of Theorems \ref{BCth9}, \ref{BCth10} also give us the following statement on the degrees of elements in a homogeneous free basis of the arrangement $cBCCat$.
\begin{proposition}\label{propdegbc}
The exponents of the arrangement $cBCCat$ are the exponents of the reflection arrangement $B_N$ with multiplicity $\widetilde m$ together with 1.
\end{proposition}


\begin{example}
Let us consider quasi-invariants $Q_{m}^V(B_2)$, where the multiplicity function $m(e_i)=2$, $m(e_i\pm e_j)=1$. Free basis  over $B_2$-invariants is given by polynomials
$$
p_1=3 x_1^7 - 7 x_1^5 x_2^2, \quad q_1=5 x_1^9 - 9 x_1^7  x_2^2
$$
and $p_2=w(p_1), q_2=w(q_1)$, where $w$ is an elementary transposition $(1,2)$. The corresponding basis of the logarithmic vector fields is given by
$$
\theta_1=p_1 \partial_1 + p_2 \partial_2, \quad
\theta_2=q_1 \partial_1 + q_2 \partial_2. 
$$
The corresponding non-homogeneous quasi-invariants $Q_m^{tr, V}$ have the basis
$$
p_1'=3 x_1^7 - 7 x_1^5 x_2^2
 - 14 x_1^5 + 35 x_1^3 x_2^2 + 7 x_1^3 - 28 x_1 x_2^2 +
 4 x_1, 
 $$
 $$
 q_1'=5 x_1^9 - 9 x_1^7  x_2^2
 -42 x_1^7 + 63 x_1^5 x_2^2 + 105 x_1^5 - 126 x_1^3 x_2^2 - 68 x_1^3 + 
 72 x_1 x_2^2
$$
and $p_2'=w(p_1'), q_2'=w(q_1')$.
The corresponding basis of the logarithmic vector fields for the extended Catalan arrangement is given by
$$
\theta_1'=p_1' \partial_1 + p_2' \partial_2, \quad
\theta_2'=q_1' \partial_1 + q_2' \partial_2. 
$$
We also have non-homogeneous quasi-invariants $Q_{1}^{tr, V}(BC_2)$, where all three multiplicity parameters are equal to 1. This space has a free basis given by
$$
{\widetilde p}_1 =3 x_1^7 - 7 x_1^5 x_2^2+
\frac14  (-35 x_1^5 + 35 x_1^3 x_2^2 + 28 x_1^3 - 7 x_1 x_2^2 - 5 x_1),
$$ 
$$
{\widetilde q}_1 = 5 x_1^9 - 9 x_1^7  x_2^2 +
\frac14  (-57 x_1^7 - 7 x_1^5 x_2^2 + 49 x_1^5 + 56 x_1^3 x_2^2 - 13 x_1^3 - 13 x_1 x_2^2 + x_1),
$$ 
and ${\widetilde p}_2=w({\widetilde p}_1), {\widetilde q}_2=w({\widetilde q}_1)$.
The corresponding basis of the logarithmic vector fields is given by
$$
{\widetilde \theta}_1= {\widetilde p}_1 \partial_1 + {\widetilde p}_2 \partial_2, \quad
{\widetilde \theta}_2 = {\widetilde q}_1 \partial_1 + {\widetilde q}_2 \partial_2. 
$$
\end{example}

\section{Concluding Remarks}

We believe that similar relations between logarithmic forms, logarithmic polyvector fields and  quasi-invariants in suitable representations hold. We hope to explore this elsewhere.

\section{Acknowledgements}

We would like to thank the organisers of the programme ``Perspectives in Lie theory" for their kind hospitality; a key part of this work was done during the visit of three of us to the programme in Pisa in 2015. 
M.F. is grateful to O. Chalykh for helpful clarifications on \cite{BC}, and we are grateful to C. Stump, G. R\"ohrle and M. Vrabec for useful discussions. 
T. A. and M. Y. are 
partially supported by JSPS KAKENHI Grant Numbers JP16H03924, JP18H01115. 
We also acknowledge support by the Research Institute for Mathematical Sciences, an
International Joint Usage/Research Center located in Kyoto University, which we visited in 2018.


\begin{thebibliography}{99}


\bibitem{ARSY}
T. Abe, G. R\"{o}hrle, C. Stump, M. Yoshinaga, 
{\it A Hodge filtration of logarithmic vector fields 
for well-generated complex reflection groups}, 
arxiv:1809.05026 (2018). 

 \bibitem{AT}
 T. Abe, H. Terao, {\em A primitive derivation and logarithmic differential forms of Coxeter arrangements}, Math. Z. 264 (2010), no. 4, 813--828.



 \bibitem{AT2}
 T. Abe, H. Terao, {\em The freeness of Shi--Catalan arrangements}, European J. Combin., 32, (2011), no. 8, pp. 1191--1198.

\bibitem{ATWaka}

T. Abe, H. Terao, A. Wakamiko, {\it Equivariant multiplicities of Coxeter
arrangements and invariant bases}, Adv. Math. 230 (2012), no.4-6, 2364--2377.


 \bibitem{ATWake}
 T. Abe, H. Terao, M. Wakefield, {\em The characteristic polynomial of a multiarrangement},  Adv. Math. 215 (2007), no. 2, 825--838.

\bibitem{BM}
J. Bandlow, G. Musiker,  {\em A new characterization for the $m$-quasiinvariants of $S_n$ and explicit basis for two row hook shapes}, 
J. Combin. Theory Ser. A 115 (2008), no. 8, 1333--1357.


\bibitem{BC}
Yu. Berest, O. Chalykh, {\em Quasi-invariants of complex reflection groups}, Compos. Math. 147 (2011), no. 3, 965--1002.

\bibitem{BEG}

Yu. Berest, P. Etingof, V. Ginzburg, {\em Cherednik algebras and differential operators on quasi-invariants}, Duke Math. J., V. 118, Nu. 2 (2003), 279--337.


\bibitem{Chalykh}

O.A. Chalykh, {\it Bispectrality for the quantum Ruijsenaars model and its integrable deformation}, J. Math. Phys. 41, Nu. 8, (2000), 5139--5167.

\bibitem{CSV}
O.A. Chalykh, K.L. Styrkas, A.P. Veselov,
{\em Algebraic integrability for the Schr\"odinger equation, and groups generated by reflections}, Theoret. and Math. Phys. 94 (1993), no. 2, 182--197.

\bibitem{CV}
O.A. Chalykh, A.P. Veselov, 
{\em Commutative rings of partial differential operators and Lie algebras}, Comm. Math. Phys.
V. 126, Nu. 3 (1990), 597--611.

\bibitem{EGsympl}
P. Etingof, V. Ginzburg, {\em Symplectic reflection algebras, Calogero-Moser space, and deformed Harish-Chandra homomorphism}, Invent. Math., 147, (2002), 243--348.



\bibitem{EG}
P. Etingof, V. Ginzburg, {\em On $m$-quasi-invariants of a Coxeter group}, Mosc. Math. J. 2 (2002), no. 3, 555--566. 


\bibitem{ERF}
P. Etingof, E. Rains, with an appendix by M. Feigin, {\it On Cohen-Macaulayness of algebras generated by generalized power sums}, Comm. Math. Phys., 347(1), (2016), pp. 163--182.


\bibitem{Fdih}
 M.V. Feigin, {\it Quasi-invariants of dihedral systems}, Math. Notes 76 (2004), no. 5-6, pp. 723--737.
 
\bibitem{FeiV}
M. Feigin, A.P. Veselov, {\em Quasi-invariants of Coxeter groups and m-harmonic polynomials}, IMRN, (2002) 10,  pp. 521--545

\bibitem{FVr}
M. Feigin, M. Vrabec, {\em Bispectrality of $AG_2$ Calogero--Moser--Sutherland system}, arXiv:2204.03677

\bibitem{FV}
G. Felder, A.P. Veselov, {\em Action of Coxeter groups on $m$-harmonic
polynomials and KZ equations},  Mosc. Math. J. 3 (2003), no. 4, 1269--1291.

\bibitem{GGOR}
V. Ginzburg, N. Guay, E. Opdam, R. Rouquier, {\em On the category O for rational Cherednik algebras}, Invent. Math. (2003), 154, 617--651.


\bibitem{Heckman}
G. J. Heckman,{\em A remark on the Dunkl differential-difference operators}, Harmonic analysis on reductive groups (Brunswick, ME, 1989), 181--191, Progr. Math., 101, Birkh\"auser Boston, Boston, MA, 1991.

\bibitem{HMRS}
T. Hoge, T. Mano, G. R\"ohrle, C. Stump, {\em Freeness of multi-reflection arrangements via primitive vector fields},  Advances in Math. 350 (2019), 63--96.





\bibitem{LT}
G. Lehrer, D. Taylor {\it Unitary reflection groups}, Australian Mathematical Society Lecture Series 20, CUP, (2009).

\bibitem{Ma}
G. Malle, {\em On the rationality and fake degrees of characters of cyclotomic algebras}, J. Math. Sci. Univ. Tokyo 6 (1999), 644-677.




\bibitem{Op}
E. Opdam,
 {\it Part II: Complex Reflection Groups and Fake Degrees. Lecture Notes on Dunkl Operators for Real and Complex Reflection Groups}, 63--90, The Mathematical Society of Japan, Tokyo, Japan, (2000).

\bibitem{OS}
P. Orlik, L. Solomon, {\em Unitary reflection groups and cohomology}, Invent. Math. 59, 77-94 (1980).


\bibitem{OT}

P. Orlik, H. Terao  {\it Arrangements of hyperplanes}, Grundlehren der mathematischen Wissenschaften 300, Springer-Verlag Berlin Heidelberg (1992).


\bibitem{Sa1} K. Saito,
{\it On the uniformization of complements of 
 discriminant loci}. In: 
 \textit{Conference Notes. Amer. Math. Soc. Summer 
 Institute, Williamstown} (1975). 
 
 
 \bibitem{Sacriterion} K. Saito, 
{\it Theory of logarithmic differential forms and logarithmic vector fields}. \textit{J. Fac. Sci. Univ. Tokyo Sect. IA Math}. 
\textbf{27} (1980), no. 2, 265--291.


 \bibitem{Sa2} K. Saito, 
{\it Uniformization of the orbifold of a finite reflection group.}
 \textit{RIMS preprint} \textbf{1414} (2003). 

\bibitem{ST}
L. Solomon and H. Terao, 
{\it The double Coxeter arrangement}. 
\textit{Comment. Math. Helv}. \textbf{73} (1998), no. 2, 237--258.


\bibitem{T80}
H. Terao, {\em Free arrangements of hyperplanes and unitary reflection groups}, Proc. Japan Acad., 56, Ser. A (1980)

\bibitem{T}
H. Terao, {\em Multiderivations of Coxeter arrangements},  Invent. Math. 148, 659--674 (2002).





 \bibitem{T05} H. Terao, {\it The Hodge filtration and the contact-order filtration of
 derivations of Coxeter arrangements.}  \textit{Manuscripta Math.} \textbf{118} (2005), 1--9. 

\bibitem{Vrabec}
M. Vrabec, {\it Generalized quantum Calogero--Moser--Sutherland systems}, Master Thesis, School of Mathematics and Statistics, University of Glasgow (2020).

\bibitem{Y}
M. Yoshinaga, {\em The primitive derivation and freeness of multi-Coxeter arrangements}, Proc. Japan Acad., 78, Ser. A (2002)


\bibitem{YC at}
M. Yoshinaga, 
{\it Characterization of a free arrangement and
 conjecture of
 Edelman and Reiner.} \textit{Invent. Math.} \textbf{157} (2004), no. 2,
 449--454.



\bibitem{Z}
G. Ziegler {\em Multiarrangements of hyperplanes and their freeness},  Singularities (Iowa City, IA, 1986), 345--359, Contemp. Math., 90, Amer. Math. Soc., Providence, RI, 1989


\end{thebibliography}
\end{document}